\documentclass[12pt,a4paper,twoside]{amsart}
\usepackage{amsmath}
\usepackage{amssymb}
\usepackage{amsfonts}
\usepackage{a4}

\newtheorem{theorem}{Theorem}[section]
\newtheorem{lemma}[theorem]{Lemma}

\theoremstyle{definition}
\newtheorem{definition}[theorem]{Definition}
\newtheorem{claim}[theorem]{Claim}

\newtheorem{remark}[theorem]{Remark}

\newcommand{\mC}{{\mathbb C}}

\newcommand{\mE}{{\mathbb E}}
\newcommand{\mF}{\mathbb F}

\newcommand{\mQ}{\mathbb Q}

\newcommand{\mZ}{{\mathbb Z}}

\newcommand{\mcO}{\mathcal O}

\newcommand{\ti}{\tilde}

\newcommand{\emp}{\emptyset}

\newenvironment{pf}{\proof[\proofname]}{\endproof}

\usepackage{color}
\begin{document}
\title{On  ranks of polynomials.}
\begin{abstract}
Let $V$ be a vector space over a field $k, P:V\to k, d\geq 3$. 
We show the existence of a function $C(r,d)$ 
such that $rank (P)\leq C(r,d)$ 
 for any field $k,char (k)>d$, a  finite-dimensional  $k$-vector space $V$
and a polynomial $P:V\to k$ of degree $d$ such that $rank(\partial P/\partial  t)\leq r$ for all $t\in V-0$.
 Our proof of this theorem is based on the application of results on Gowers norms for finite fields $k$.
We don't know a direct proof even in the case when $k=\mC$.

\end{abstract}

\author{David Kazhdan}
\address{Einstein Institute of Mathematics,
Edmond J. Safra Campus, Givaat Ram 
The Hebrew University of Jerusalem,
Jerusalem, 91904, Israel}
\email{david.kazhdan@mail.huji.ac.il}

\author{Tamar Ziegler}
\address{Einstein Institute of Mathematics,
Edmond J. Safra Campus, Givaat Ram 
The Hebrew University of Jerusalem
Jerusalem, 91904, Israel}
\email{tamarz@math.huji.ac.il}

\dedicatory{Dedicated to A. Kirillov  on the occasion of his 80th birthday}

\thanks{The second author is supported by ERC grant ErgComNum 682150}

\maketitle

\section{introduction}
We fix $d\geq 0$ and restrict our attention to field $k$ such that $d!\neq 0$ (in other words when  either $k$ is a field of characteristic zero or $char(k)>d$). 
For any $k$-valued function $f$ on $V$ and $t\in V-0$ we define  $\Delta_tf(x) := f(x+t)-f(x)$.
For any 
$t\in V$ we denote by $P\to \partial P/ \partial t$ the differentiation of the ring of polynomial function on $V$ such that $ \partial P/ \partial t= P(t)$ if $P:V\to k$ is a linear function.%
\begin{definition} [Algebraic rank filtration] Let $k$ be a field. \\
a) For a homogeneous  polynomial  $P$ on a
 finite-dimensional  $k$-vector space $V$,   of degree $d\geq 2$ we
define the {\it rank} $r(P)$ of $P$ as  the minimal number $r$ such that it is possible to write $P$ in the form
$$P = \sum ^ r_{i=1} L_iR_i,$$ 
where $L_i, R_i\in \bar k[V^\vee ]$ are homogeneous polynomials of positive degrees. \\
b) For any   polynomial  $P$ on $V,t\in V-0$ we define 
$P_t:=\partial P/\partial  t$. \\
c) If $P:V\to k$ is any polynomial of degree $d$  we define 
$r(P)=r(P_d)$ where $P_d:V\to k$ the homogeneous part of $P$ of the degree $d$.
\end{definition}
\begin{theorem}[Main]\label{main}
Let $d, r >0$. There exists a function $C(r,d)$ such that $r(P)\leq C(r,d)$ for any polynomial   $P$ of degree $d$ on  a
 finite-dimensional  $k$-vector space $V$ such that $r( \Delta_tP)\leq r$ for all $t\in V$.

\end{theorem}

\begin{remark}
For any $k$-valued function $f$ on $V$ and $t\in V-0$ we define 
 $\Delta_tf(x) := f(x+t)-f(x)$. If $f$ is a polynomial of degree $d$ then for generic 
$t\in V$ both $f_t$ and $\Delta _t(f)$ are polynomials of degree $d-1$ and the difference 
h $f_t-\Delta _t(f)$ is of degree $\leq d-2$. Therefore  $r(P_t)=r(\Delta_tP)$ 
 for any polynomial $P,t\in V$  and generic $t\in V$. 
In other words, the condition $r(P_t)\leq r, t\in V-0$ is equivalent to the condition $r(\Delta _t(P))\leq r, t\in V$. This is the 
\end{remark}
We grateful for Professor Jan Draisma who informed us about Theorem $1.8$ of \cite{DES} where a stronger result is proven for the case of cubic polynomial.
\section{The case of finite fields}
In this section we assume that $k$ is a finite field of characteristic $>d$.
\begin{definition}
\begin{enumerate}
\item
We denote by $e_p:\mF _p\to \mC ^\star$ the additive character 
$e_p(x):=\exp (\frac {2\pi ix}{p})$. For any finite field $\mF _q$, $q=p^l$ we define $\psi :=e_p(tr _{\mF _q/\mF _p}(x))$ where 
$tr _{\mF _q/\mF _p}:\mF _q\to \mF _p$ is the trace map. \\
\item
  For a function $f$ on a finite set $X$ we define
\[
\mE_{x \in X}f(x) = \frac{1}{|X|}\sum_{x \in X}f(x).
\]
We use $X \ll_L Y$ to denote the estimate 
$|X| \le C(L) |Y|$, where the constant $C$ depends only on $L$.  \\
\item Let $k=\mF _q$ be a finite field, $V$ be a finite dimensional $k$-vector space.  Given a function  $F: V \to k$ the  $m$-th {\em Gowers norm}  of $F$  is defined by
\[
\|\psi(F)\|_{U_m}^{2^m} = \mE_{v, v_1, \ldots, v_m \in V} \psi(\Delta_{v_m} \ldots  \Delta_{v_1} F(v)).
\]
These were introduced by Gowers in \cite{G}, and were shown to be norms for $m>1$.\\
\item For a homogeneous polynomial $P$ on $V$ of degree $d$ we define
\begin{equation}\label{multilinear}
\tilde P(x_1, \ldots, x_d) = \Delta_{x_d}\ldots \Delta_{x_1}P(x).
\end{equation}
This is a multilinear homogeneous form in $x_1, \ldots, x_d \in V$ such  that 
 $$P(x) = \frac{1}{d!}\tilde P(x,\ldots, x).$$  
\end{enumerate}
\end{definition}

\begin{claim}\label{g}There exists a function $s(r,d)$ such
 that for any
 finite field $k$, a finite-dimensional
 $k$-vector space $V$ and  a multilinear homogeneous polynomial 
 $P:V^d \to k$ of degree $d \ge 2$ and rank $\leq r$ we have 
$\mE_{ x \in V} \psi(P(x))\ge  s(r,d)$.
\end{claim}

This Claim is Lemma 2.2 in \cite{KZ}.

\begin{lemma}\label{previous} There exists a function $b(r,d)$ such  that for any
 finite field $k$, a finite-dimensional
 $k$-vector space $V$ and  a polynomial $P:V\to k$ of degree $d$ 
and rank $\leq r$ we have  $ \|\psi (P)\|_{U_{d}}>b(r,d)$. 
\end{lemma}
\begin{proof} 
\[ \|\psi (P)\|_{U_{d}}^{2^d} = \mE_{h_1, \ldots h_d, x \in V} \psi (\Delta_{h_d} \ldots \Delta_{h_1}P)\]
 Since $P$ is of degree $d$  the function $G(h_1, \ldots, h_d)=\Delta_{h_d} \ldots \Delta_{h_1}P(x)$ is a constant equal to $\ti P (h_1,...,h_d)$.

Since $P$ is of rank $<r$, it follows from   \ref{g} that 
 \[
 \mE_{h_1, \ldots h_d} G(h_1, \ldots, h_d) \ge c(r,d)
 \]

\end{proof}

\begin{theorem}\label{f} There exists a function $C(r,d)$ such  that $r(P)\leq C(r,d)$ for any polynomial $P:V \to k$ of degree $d$ such that $r(\Delta _t(P))\leq r$  for all $t \in V$.

\end{theorem}

\begin{proof} Let $b(r,d-1)$ be as in the Lemma \ref{previous}.
We have 
\[
\|\psi (P)\|^{2^d}_{U_d} = \mE_t \|\psi (P(x+t)-P(x))\|^{2^{d-1}}_{U_{d-1}}=\mE_t \|\psi (P_t)\|^{d-1}_{U_{d-1}} \geq b^{d-1}(r,d-1) >0
\]
for some constant since $P_t$ is of rank $<r$. 

Therefore (see \cite{BL})
 there exists $C(r,d)$ such  that  $r(P)\leq C(r,d)$.

\end{proof}
\section{Proof of the Main Theorem}

We start with the statement of the  following well known result. Let $S, T$ be finite sets and let 
 $(\star)$ be a system of equations of the form  $ A_s(x_ i)=0,s\in S,  B_t(x_i)\neq 0,t\in T$ where $A_s,B_t \in \mZ [x_1,\ldots,x_n]$. For any field $k$ we denote by $Z(k)\subset k^n$ the subset of solutions of the system $(\star)$.

\begin{claim}\label{claim} Assume that $Z (\mF _q)=\emp$ for all $q=p^l,p\geq d$. Then $Z(k)=\emp$ if $k$ is either a field of characteristic $\geq d$ or a field of characteristic zero.
\end{claim}

\begin{proof}
It is sufficient to consider the case when the field $k$ is algebraically closed.
Consider first the case when $k$ is a field of positive characteristic $p$. By the assumption $Z(\bar \mF _p)=\emp$ where $\bar \mF _p$ is the algebraic closure of $\mF _p$. Then the result follows from Corollary 3.2.3 of \cite{M}. 

To deal with the case of fields of characteristic $0$ we choose a non-primitive ultrafilter  $D$ on the set of prime numbers and define $K=\prod \bar \mF _p/D$ 
(see the section $2.5$ in \cite{M}). Then $K$ is an algebraically closed field of characteristic $0$ and the Loe's theorem imply that $Z(K)=\emp$. As before  Corollary 3.2.3 of \cite{M} implies that  $Z(k)=\emp$ for all   algebraically closed fields of characteristic zero.

Since some people are not familiar with Model theory
we  indicate another way to prove that  the conditions of Claim \ref{claim} imply that $Z(\mC)=\emp$.
To simplify the arguments we assume that polynomials $A_s$ and $B_t$ are homogeneous.

As follows from the Hilbert's Nullstellensatz theorem it is sufficient to show that $Z(\bar \mQ) =\emp$. 
We show that an assumption 
$$\exists (x_1,\ldots,x_ n)\in Z(\bar \mQ)$$
 leads to a contradiction.

Let $K\subset \bar \mQ$ be the field generated by elements
 $x_i,1\leq i\leq n$ and $\mcO \subset K$ be the ring of integers. Since polynomials  $A_s$ and $B_t$ are homogeneous we have $(cx_1,\ldots,cx_ n)\in Z(\bar \mQ)$ for any $c\in K^\star$.
So we can assume that  $x_i\in \mcO,1\leq i\leq n$.

Let $N_{K/\mQ }:K^\star \to \mQ ^\star$ be the norm map.
Choose a prime number $p\geq d$ such that  
$N_{K/\mQ }(B_t(x_1,\ldots,x_n))$  are prime to $p$ for all $t\in T$.
Let $I\subset \mcO$ be any maximal ideal containing $p\mcO$.
Then $(\bar x_1,\ldots,\bar x_ n)\in (\mcO /I)^n$ is a point in $Z(\mcO /I)$. But by the assumption $Z(\mcO /I)=\emp$. This contradiction
proves the claim.
\end{proof}

Now we can prove the Theorem \ref{main}.
\begin{pf}
Fix $r$ and $d$ and for any $N\geq 1$ consider the algebraic variety $X_N$ of
 polynomials $P(x_1,\ldots,x_N)=\sum _{\bar i\in I}c_{\bar i}x^{\bar i}$ where where 
$$I=\{ i_1,\ldots,i_N\geq 0| \sum _ji_j=d\}, \ x^{\bar i}:=x_1^{i_1}...x_N^{i_N}.$$
 The condition  $r(P)>C(r,d)$ 
 can be written as a
 system 
$\{ A_s(c_{\bar i})=0\},1\leq s\leq A$ and  $\{ B_t(c_{\bar i})\neq 0\}, 1\leq t\leq B$, where $A_s, B_t$ are homogeneous polynomials in $c_{\bar i}$ with integer coefficients.  

On the other hand it follows from Theorem $3.2.2$ in \cite{M} (or from the theorem of Chevalley) 
that the  condition  
$$\forall t: r(P_t)\leq r$$ can  be written as a
 system equations of $c_{\bar i}$  the form 
$\{ C_u(c_{\bar i})=0\},1\leq u\leq C$
 and  $\{ D_v(c_{\bar i})\neq 0\}\}, 1\leq v\leq D$, where $C_u,D_v$ are homogeneous polynomials in $c_{\bar i}$ with integer coefficients.

 Let $W$ be the vector space with coordinates $c_{\bar i}$ and 
 $Z_N\subset W$ be the algebraic variety defined by the system of algebraic equalities
$$\{ A_s(c_{\bar i})= C_u(c_{\bar i})=0\},$$
and  inequalities
 $$\{ B_t(c_{\bar i})\neq 0,  D_v(c_{\bar i})\neq 0\}.$$
By definition,  for any field $k$ the points
 $c_{\bar i}\in Z_N(k)$ are in bijection with polynomials  $P=\sum _{\bar i}c_{\bar i}x^{\bar i}$ of degree and rank $>C(r,d)$ such that $r(\partial P/\partial t)\leq r$. 

So Theorem \ref{f} implies that $Z_N(\mF _q)=\emp$ for any $q=p^l,p>d$. 
Therefore by Claim \ref{claim} we have
$Z_N(k)=\emp$ for any $N$ and  any field $k$ of characteristic $0$ or of characteristic $>d$.

 Theorem \ref{main} is  proven.
\end{pf}

\end{document}